\documentclass{article}
\usepackage[T1]{fontenc}
\usepackage[utf8]{inputenc}
\usepackage[english]{babel}
\usepackage{amsthm,amssymb,amsmath}
\usepackage{hyperref}

\newtheorem{Theorem}{Theorem}[section]
\newtheorem{Lemma}[Theorem]{Lemma}

\newtheorem{athm}{Theorem}

\theoremstyle{definition}

\newtheorem{question}{Question}

\theoremstyle{remark}
\newtheorem{Note}{Remark}

\newcommand{\Z}{\mathbb{Z}}

\DeclareMathOperator{\Aut}{Aut}

\DeclareMathOperator{\opp}{\mathrm{opp}}
\newcommand{\bigcdot}{\boldsymbol{\cdot}}

\title{On a Cauchy theorem for finite skew braces}

\author{M. Damele \thanks{Dipartimento di Matematica, Università di Cagliari, Via Ospedale 72, 09124 Cagliari, Italy; \texttt{marco.damele@unica.it}; ORCID 0009-0008-3088-5766}
\and V. P\'erez-Calabuig \thanks{Departament de Matem\`atiques, Universitat de Val\`encia, Dr.\ Moliner, 50, 46100 Burjassot, Val\`encia, Spain; \texttt{Vicent.Perez-Calabuig@uv.es}; ORCID 0000-0003-4101-8656.}}

\date{}

\begin{document}
\maketitle

\begin{abstract}
One of the major problems in the structural theory of skew braces consists in the classification of skew braces of finite order up to isomorphism. In this light, the open question of the existence of a Cauchy theorem for finite skew braces is of great interest. We prove a positive answer for the classes of finite two-sided skew braces and bi-skew braces. Consequences and related structural results are also outlined.
\end{abstract}



\section{Introduction}

Skew braces are a novel algebraic structure which was introduced in the context of classifying non-degenerate set-theoretical solutions of the Yang-Baxter equation (YBE, for short). A \emph{(left) skew brace} is a set $B$ endowed with two group structures, $(B,+)$ and $(B,\cdot)$ ---not necessarily abelian---, such that the following (left) compatibility condition holds: 
\begin{equation}
\label{eq:left-dist}
a \cdot (b+c) = a \cdot b - a + a \cdot c, \quad \text{for all $a,b,c \in B$,}
\end{equation}
---as usual, products are denoted by juxtaposition and act before sums. One of the major problems in the theory, with strong implications in the classification problem of solutions, is the classification up to isomorphism of skew braces of finite order (see~\cite{AcriBonatto20-algcoll-ab, Acri_2020, AcriBonatto22-jaa-nonab}, for example). The order of a skew brace is defined as the order of any of its component groups. \emph{In this work, every skew brace $B$ is considered finite and we write by $|B|$ its order.}

It is out of question that the structural analysis of every algebraic structure undoubtedly passes through the description of its inner substructures. A subset $S$ of a skew brace $B$ is said to be a \emph{subbrace} if both $(S,+)$ and $(S,\cdot)$ are subgroups of $(B,+)$ and $(B,\cdot)$, respectively.
Bearing in mind that the Cauchy theorem is one of the milestones of the classification of finite groups, the following question consequently arises.

\begin{question} 
\label{que:Cauchy} Let $B$ be a skew brace. Is it true that for every prime divisor $p$ of $|B|$ there exists a subbrace $S$ of $B$ with $|S| = p$?
\end{question}

The problem of studying the subbrace structure of a given skew brace is of great difficulty ---e.g. the cases of the description of skew braces without subbraces in~\cite{ballesterbolinches2024solubleskewleftbraces} or subbraces generated by one element in~\cite{BallesterEstebanKurdachenkoPerezC-25-actions, Rump20}. Therefore, it is reasonable to handle first this question for some worthy classes of skew braces with deep implications in the structural theory. In this sense, it stands out the class of the so-called two-sided skew braces, i.e. skew braces for which the compatibility condition~\eqref{eq:left-dist} holds on both sides. Two-sided skew braces were firstly studied in~\cite{nasybullov2019}, and they generalise Jacobson radical rings (see~\cite{RUMP2007153});
moreover, in~\cite{nasybullov2019} a negative answer is given to question~19.90-c) of the Kourovka notebook~\cite{Kourovka26} for two-sided skew braces. This is equivalent to a positive answer of Byott's conjecture in terms of regular subgroups of the holomorph of a group (see~\cite{byott2015} and~\cite{Trappeniers_2023}), which is one of the most challenging open problems in the theory (see~\cite{Byott24-insoluble, DameleLoi25, Tsang_2019} for some other partial results). Another remarkable class of skew braces is the class of \emph{bi-skew braces}, i.e. skew braces such that interchanging the role of group operations gives rise also to a skew brace structure. Bi-skew braces were introduced in~\cite{childs19} in the context of the Hopf-Galois extension theory where they play a key role.
In~\cite{biSkewbrace} there is also a negative answer to question~19.90-c) for the class of bi-skew braces. 

The main results of this work give a positive answer to Question~\ref{que:Cauchy} for the classes of two-sided skew braces and bi-skew braces.
\begin{athm}
\label{teo:two-sided}
Let $B$ be a two-sided skew brace. For every prime divisor $p$ of $|B|$, there exists a subbrace $S$ of $B$ with $|S| = p$.
\end{athm}

\begin{athm}
\label{teo:bi-skew}
Let $B$ be a bi-skew brace. For every prime divisor $p$ of $|B|$, there exists a subbrace $S$ of $B$ with $|S| = p$.
\end{athm}
The proof of the two-sided case strongly depends on Lemma~\ref{SolvSB}: Question~\ref{que:Cauchy} holds for soluble skew braces with nilpotent multiplicative group (see Section~\ref{sec:prelim} for the definition of soluble skew braces).

\section{Preliminaries}
\label{sec:prelim}
Let $B$ be a skew brace. Equation~\eqref{eq:left-dist} yields the same identity element in both additive and multiplicative groups. We denote it by~$0$. Then, the multiplicative group acts on the additive group by means of a homomorphism $\lambda\colon a \in (B,\cdot) \mapsto \lambda_a\in \Aut(B,+)$, with $\lambda_a(b) = -a+ab$ for every $a,b\in B$. Moreover, equation~\eqref{eq:left-dist} on the left (resp. on the right) yields
\begin{equation} 
\label{eq:a(-b)}
a(-b) = a -ab +a, \quad (-b)a= a -ba +a, \quad \forall a,b\in B
\end{equation}

A subbrace $I$ of $B$ is called an \emph{ideal}, if $I$ is $\lambda$-invariant, and it is a normal subgroup of both $(B,+)$ and $(B,\cdot)$. Ideals allow to consider quotient skew braces $(B/I,+,\cdot)$, with $a+I = aI$ for every $a\in B$. As in case of rings, a skew brace $B$ is said to be \emph{simple} if it has exactly two ideals: $\{0\}$ and~$B$.

A skew brace $B$ is called \emph{trivial} if the $\lambda$-action is trivial, i.e. $ab = a+b$ for every $a,b\in B$. It turns out that a skew brace of order prime $p$ is a trivial skew brace with equal groups isomorphic to a cyclic group of order~$p$.

Our first lemma is a known result (see~\cite{dededind-skew-brace25}), which gives a sufficient condition for a cyclic subgroup of $(B,+)$ to be a subbrace. We include a proof for the sake of completeness.

\begin{Lemma} \label{easyLemma}
Let $B$ be a skew brace and $x \in B$. If $\lambda_{x}(x) = x$, then $\langle x \rangle_{+}$ is a trivial subbrace of~$B$.
\end{Lemma}
\begin{proof}
Since $\lambda_{x}(x) = x$, we also have that $\lambda_{x^{-1}}(x) = x$. Thus, $0 = x^{-1}x = x^{-1} + \lambda_{x^{-1}}(x) = x^{-1}+x$, i.e. $x^{-1} = -x$. By induction, assume that $x^k = kx$ and $x^{-k} = -kx$ for some $k\geq 1$. Then,
\[ x^{k+1} = xx^{k} = x + \lambda_x(x^k) = x + \lambda_x(kx) = x + k\lambda_x(x) = (k+1)x,\]
and, analogously,
\[x^{-(k+1)} = x^{-1}x^{-k} = x^{-1} + \lambda_{x^{-1}}(x^{-k}) = -x + \lambda_{x^{-1}}(-kx) = -(k+1)x\]
Hence $\langle x \rangle_{+}$ is a trivial subbrace of~$B$.
\end{proof}

Given a skew brace $(B,+,\cdot)$, one can replace the additive group $(B,+)$ by its opposite group, i.e. $(B,+^{\mathrm{opp}})$ with $a+^{\mathrm{opp}}b= b+a$ for every $a,b\in B$, to obtain a new skew brace $B^{\mathrm{opp}} = (B,+^{\mathrm{opp}},\cdot)$. It is called the \emph{opposite skew brace} of~$B$ (see \cite{koch2020oppositeskewleftbraces}). We denote the $\lambda$-map of an element $a \in B^{\mathrm{opp}}$ as~$\lambda^{\mathrm{opp}}_a$, and for every $b\in B$, it holds
\begin{equation}
\label{eq:lambdaopp}
\lambda^{\mathrm{opp}}_a(b) = ab - a = a + \lambda_a(b) - a.
\end{equation}
As a consequence, ideals of $B$ and $B^{\mathrm{opp}}$ coincide. The opposite skew brace $B^{\opp}$ of a trivial skew brace $B$ is also known as an \emph{almost trivial} skew brace, where $ab = b+a$ for every $a,b\in B$.

\begin{Note}
\label{Note:trivialSemitrivial}
If a skew brace $B$ is either trivial or almost trivial, then every subset $H$ of $B$ satisfies: $H \le (B,+)$ if and only if  $H \le (B,\cdot)$. Thus, for every prime $p$ dividing $|B|$, there exists a subbrace $H \le B$ such that $|H| = p$.
\end{Note}

Question~\ref{que:Cauchy} has a positive answer if the additive group of a skew brace is nilpotent.
\begin{Lemma} \label{nilptype}
Let $B$ be a skew brace with nilpotent additive group. Then, for every prime $p$ dividing $|B|$, there exists a subbrace $H \le B$ such that $|H| = p$.
\end{Lemma}

\begin{proof}
Let $B$ be a minimal counterexample, and let $p$ be a prime dividing $|B|$, for which there is no subbrace of order~$p$ in~$B$. Let $(P,+)$ be the Sylow $p$-subgroup of $(B,+)$. Since $(P,+)$ is a characteristic subgroup of $(B,+)$, $(P,+)$ is a subbrace of~$B$. By the minimality of $B$, $(P,+)$ cannot be strictly contained in $(B,+)$. Thus, $B$ has order a power of~$p$. 

If $B$ does not contain any proper subbrace, then by \cite[Theorem~A]{ballesterbolinches2024solubleskewleftbraces}, 
$B$ must be the trivial brace of order~$p$, a contradiction. Hence, there exists a proper subbrace $H \subsetneq B$, and by the minimality of $B$, we arrive to a final contradiction.
\end{proof}

The so-called \emph{star product} of a skew brace $B$ can be seen as a measure of the difference between products and sums in~$B$: $a \ast b = -a + ab -b = \lambda_a(b) - b$, for every $a,b\in B$. It follows that $I$ is an ideal of $B$ if, and only if, $B\ast I, I \ast B \subseteq B$. The following easy to check properties are satisfied for every $a,b,c\in B$:
\begin{align}
a \ast(b+c) & = a\ast b + b + a \ast c - b \label{eq:dist-sum}\\
a \ast (-b) & =-b-a\ast b+b \label{eq:a*(-b)}\\
\lambda_a(b\ast c) & = \lambda_{aba^{-1}}(\lambda_a(c)) - \lambda_a(c) = (aba^{-1})\ast \lambda_a(c)\label{eq:ast-prod}
\end{align}
Given subsets $X, Y$ of a skew brace~$B$, $X \ast Y$ is defined as the additively generated subgroup $X \ast Y:= \langle x \ast y \mid x\in X, y \in Y\rangle_+$. If $I$ is an ideal of $B$, then clearly $(a+I)\ast (b+I) = (a\ast b)+I$ for every $a,b\in B$. It follows that $B$ is trivial if, and only if, $B\ast B = 0$. Thus, $B^2:= B \ast B$ is the smallest ideal of $B$ with a trivial quotient skew brace. Analogously in $B^{\opp}$: $a \ast^{\opp} b = -b + ab -a$ for every $a,b\in B$. Then, $B$ is almost trivial if, and only if, $(B^{\opp})^2 = 0$, and $(B^{\opp})^2$ is the smallest ideal with an almost trivial quotient skew brace. 

A skew brace $B$ is said to be \emph{abelian} if it is trivial and $a+b=b+a$ for every $a,b\in B$.
\begin{Note}
\label{nota:B2-Bopp2}
Every abelian skew brace can be seen also as almost trivial. Thus, if $I$ is an ideal of $B$ such that $B/I$ has prime order, then both $(B^{\opp})^{2},B^{2} \le I$.
\end{Note}

In~\cite{ballesterbolinches2024solubleskewleftbraces}, \emph{soluble} skew braces are defined as those skew braces for which there exists a chain of ideals
\[ I_0 = 0 \subseteq I_1 \subseteq \cdots \subseteq I_n = B\]
such that $I_{k}/I_{k-1}$ is an abelian skew brace for every $1\leq k \leq n$. By \cite[Theorem~B]{ballesterbolinches2024solubleskewleftbraces}, it turns out that a minimal ideal of a soluble skew brace must be an abelian skew brace, with equal groups isomorphic to a $p$-elementary abelian group.

%
%
%
%
%
%
%
%

We end this section with the following key lemma.

\begin{Lemma} \label{SolvSB}
Let $(B,+,\cdot)$ be a soluble skew brace such that the multiplicative group $(B,\cdot)$ is nilpotent. Then, for every prime $p$ dividing $|B|$, there exists a subbrace $H \le B$ such that $|H| = p$.
\end{Lemma}

\begin{proof}
Let $B$ be a minimal counterexample and let $p$ be a prime dividing $|B|$ such that there is no subbrace of $B$ of order~$p$. Note that $B$ is not simple; otherwise, by \cite[Corollary~21]{ballesterbolinches2024solubleskewleftbraces}, we would get that $B$ is a trivial skew brace of prime order. 

Let $M$ be a minimal ideal of $B$. We know that $M$ is an abelian skew brace with equal groups isomorphic to a $q$-elementary abelian group with $|M| = q^{n}$ for some prime $q$ and positive integer $n$. Observe that $q\neq p$, since $M$ has a subbrace of order $q$ by Remark~\ref{Note:trivialSemitrivial}. Thus, $p \mid |B/M|$. Since $B/M$ is soluble with nilpotent multiplicative group, by the minimality of $B$, there exists a subbrace $P/M$ of $B/M$ of order~$p$. Again, by the minimality of~$B$, $P = B$ so that $|B| = p \cdot |M| = pq^{n}$, and $B/M$ is an abelian skew brace with equal groups isomorphic to a cyclic group of order~$p$. Since $(B,\cdot)$ is nilpotent, observe that, in particular, $(B,\cdot)$ is abelian.

Let $\langle y \rangle_{\bigcdot}$ be the Sylow $p$-subgroup of $(B,\cdot)$. Consider the action
\[
\phi: \langle y \rangle_{\bigcdot} \rightarrow \operatorname{Sym}(\operatorname{Syl}_{p}(B,+))
\]
defined by $\phi(y^k)(P) = \lambda_{y^k}(P)$, for every $P \in \operatorname{Syl}_{p}(B,+)$ and every $k\in \mathbb{Z}$. Since $\langle y \rangle_{\bigcdot}$ has order $p$, the set of fixed points of the action is congruent with $|\operatorname{Syl}_{p}(B,+)|$ modulo~$p$. Thus, there exists $P \in \operatorname{Syl}_{p}(B,+)$ such that $\lambda_{\langle y \rangle_{\bigcdot}}(P) = P$. Therefore, we can consider the restriction $\psi:= \langle y \rangle_{\bigcdot} \rightarrow \operatorname{Aut}(P,+)$, and hence, $\psi$ must be trivial.

Call $P = \langle x \rangle_{+}$, for some $x\in B \setminus M$.  Since $B/M$ is an abelian skew brace of order $p$, $x^p \in M$. Thus, $p$ divides the multiplicative order $o_{\bigcdot}(x)$, and we must have either $o_{\bigcdot}(x) = p$ or $o_{\bigcdot}(x) = pq$. In the first case, we have $\langle x \rangle_{\bigcdot} = \langle y \rangle_{\bigcdot}$, so that $x = y^k$ for some natural number $k$. Therefore,
\[
\lambda_x(x) = \lambda_{y^k}(x) = \psi(y^k)(x) = x,
\]
which implies, by Lemma~\ref{easyLemma}, that $P = \langle x \rangle_{+}$ is a trivial subbrace of order~$p$.

Suppose instead that $o_{\bigcdot}(x) = pq$. Then $\langle x^q \rangle_{\bigcdot}$ has order $p$, and therefore, $\langle x^q \rangle_{\bigcdot} = \langle y \rangle_{\bigcdot}$. Let $y = x^t$, for some multiple $t$ of~$q$. Since $(B,\cdot)$ is abelian and $\lambda_y(x) = x$, we have $xy = yx = y + x$. Then, it holds that
\begin{align*}
yx^2 &= (yx)x = (y + x)x = x(y + x) = xy - x + x^2 = y + x - x + x^2 = y + x^2.
\end{align*}
Hence, $\lambda_y(x^2) = x^2$. By induction, we obtain that for every natural number $k$, $\lambda_y(x^k) = x^k$. In particular, $\lambda_y(y) = \lambda_y(x^t) = x^t = y$. Hence, by Lemma~\ref{easyLemma}, we conclude that $\langle y \rangle_{+}$ is a trivial subbrace of $B$ of order~$p$, and we arrive to a final contradiction.
\end{proof}

\section{Proof of Theorem~\ref{teo:two-sided}}

Before proving Theorem~\ref{teo:two-sided} we collect some well-known properties of two sided skew braces. We shall use that the class of two-sided skew braces is closed by subbraces and quotients. Recall that skew braces with abelian additive group are known as braces. Braces were
introduced in~\cite{RUMP2007153} as a generalisation of Jacobson radical rings: two-sided braces \(B\) are exactly Jacobson radical rings \((B,+,\ast)\).

A first application of the distributivity law~\eqref{eq:left-dist} on both two-sides yields the following lemma.

\begin{Lemma}
\label{lema:centralsubbrace}
Let $B$ be a two-sided skew brace and let $a\in B$. Then, the multiplicative centraliser $\operatorname{C}_{(B,\bigcdot)}(a)$ is a subbrace of~$B$.
\end{Lemma}

\begin{proof}
Let $b,c\in \operatorname{C}_{(B,\bigcdot)}(a)$. By equation~\eqref{eq:a(-b)},
\[ a(b -c) = ab - a + a(-c) = ab -a + a -ac +a = ba -ca +a = ba -a + (-c)a = (b-c)a\]
Thus, $b-c\in \operatorname{C}_{(B,\bigcdot)}(a)$.
\end{proof}

Let \(B\) be a two-sided skew brace. For every $a,b,c,d\in B$, a direct calculation shows that
\begin{align}
\nonumber a(b+c)a^{-1} & = (ab - a + ac)a^{-1} \\
\nonumber &= aba^{-1} -a^{-1} +(-a)a^{-1} - a^{-1} + aca^{-1} \\
 & = aba^{-1} + aca^{-1} \label{eq:int-aut-mult}\\
(a+b)(c+d) & = ac + (b\ast^{\opp}c) + a \ast d + bd; \label{eq:dist1}\\
& =  ac + a \ast d + (b \ast^{\opp} c) + bd \label{eq:dist2};
\end{align}
where in~\eqref{eq:int-aut-mult} we use~\eqref{eq:left-dist} on both sides and~\eqref{eq:a(-b)}; in~\eqref{eq:dist1} we use~\eqref{eq:left-dist} first on the left and then on the right; and vice versa for~\eqref{eq:dist2}. From~\eqref{eq:int-aut-mult}, we obtain that characteristic subgroups of $(B,+)$ are ideals, as they are $\lambda$-invariant and normal in $(B,+)$ and $(B,\cdot)$. On the other hand, equations~\eqref{eq:dist1} and~\eqref{eq:dist2} show that \(B^{2}\) and $(B^{\mathrm{opp}})^2$ centralise each other in \((B,+)\). As a consequence, we obtain the following known lemmas (we include a proof for the sake of completeness).

\begin{Lemma}[{\cite[Corollary 4.8]{Trappeniers_2023}}]
\label{lema:simplicity}
Let $B$ be a simple two-sided skew brace. Then, $B$ is either trivial or almost trivial.
\end{Lemma}

\begin{proof}
By a way of contradiction, assume that $B$ is neither trivial nor almost trivial. Then, the simplicity of $B$ implies that $B^2 = (B^{\opp})^2 = B$. Thus, we have seen that $(B,+)$ must be abelian, and therefore, $(B,+,\ast)$ is a finite Jacobson radical ring. Finite Jacobson radical rings are nilpotent, and therefore, $B^2=  B \ast B$ must be proper in~$B$, so we arrive to a contradiction. 
\end{proof}


   



\begin{Lemma}[{\cite[Lemma~5.4]{Trappeniers_2023}}] \label{Lemma: M*M}
Let $B$ be a two-sided skew brace and take $I = \bigl(B^{2} \cap (B^{\mathrm{opp}})^{2}\bigr)$. Then, $I \ast I$ is an ideal of~$B$.
\end{Lemma}

\begin{proof}
Observe that $I$ is an ideal with abelian additive group $(I,+)$, as $B^2$ and $(B^{\opp})^2$ centralise each other in $(B,+)$. By ~\eqref{eq:ast-prod}, it follows that $B \ast (I\ast I) \subseteq I\ast I$, i.e. $I$ is $\lambda$-invariant.  Now, let $b\in B$ and let $a\ast c \in I\ast I$ with $a,c\in I$. Applying equations~\eqref{eq:a*(-b)} and~\eqref{eq:dist-sum}, observe that
\[
\begin{aligned}
a\ast(-b+c+b)
&= a\ast(-b) - b + a\ast(c+b) + b\\
&= -b - a\ast b + a\ast c + c + a\ast b - c + b\\
&= -b + a\ast c + b \in I\ast I.
\end{aligned}
\]
In the last equality we have used that $c,a\ast c, a\ast b\in I$, and therefore, they additively commute with each other. Hence, $(I,+)$ is a normal subgroup of $(B,+)$.

On the other hand, by~\eqref{eq:int-aut-mult}, it also holds 
\[
\begin{aligned}
b^{-1}(a*c)b
&= b(-a+ac-c)b^{-1}\\
&= b(-a)b^{-1} + b(ac)b^{-1} + b(-c)b^{-1}\\
&= (bab^{-1}) \ast (bcb^{-1}) \in I\ast I.
\end{aligned}
\]
Thus, $(I\ast I,\cdot)$ is normal in $(B,\cdot)$.
\end{proof}

%

\begin{proof}[Proof of Theorem~\ref{teo:two-sided}]
Let $B$ be a counterexample of minimal order, and let $p$ be a prime divisor of $|B|$ such that $B$ has no subbrace of order~$p$. 

Remark~\ref{Note:trivialSemitrivial} shows that $B$ is neither a trivial nor an almost trivial skew brace. Thus, Lemma~\ref{lema:simplicity}, $B$ can not be simple. 

Let $M$ be a minimal ideal of $B$. By the minimality of~$B$, we may assume that $p \nmid |M|$. Arguing as in the proof of Lemma~\ref{SolvSB}, we must have that $|B| = p|M|$. Then, $B/M$ is an abelian skew brace of order~$p$. Thus, the derived subgroup $(B,+)' \le M$. Since the derived subgroup is characteristic, then it is an ideal,  and by minimality of $M$, we have two possibilities: either $(B,+)' = 0$ or $(B,+)' = M$. If $(B,+)' = 0$, then $(B,+)$ is abelian, and by Lemma~\ref{nilptype} $B$ would contain a subbrace of order $p$, a contradiction. Hence, $(B,+)' = M$. Moreover, the minimality of $M$ implies that $(M,+)$  is a characteristically simple group. Thus, there exist a simple group $S$ and a positive integer $n$ such that $(M,+)$ is isomorphic to~$S^{n}$.

On the other hand, by Remark~\ref{nota:B2-Bopp2}, it holds that $B^{2}, (B^{\opp})^2 \le M$, as $B/M$ is of prime order. Thus, $B^2 = (B^{\opp})^2 = M$, as $B$ is neither trivial nor almost trivial. Since $B^2$ and $B^{\opp}$ centralises each other in $(B,+)$, it follows that $(M,+)$ is abelian. Hence, $S$ is isomorphic to a cyclic group of order prime~$q$, and $M$ is $q$-elementary abelian ---observe that $q\neq p$, otherwise we arrive to a contradiction. Moreover, it follows from Lemma~\ref{Lemma: M*M}  that $M^{2}$ is an ideal of~$B$, which is strictly contained in $M$ as $M$ can be seen as a Jacobson radical ring. Therefore, it follows that $M^2 = 0$, i.e. it is an abelian skew brace. Hence, $B$ is a soluble skew brace.

Recall that the additive group $(B,+)$ is a semidirect product of $(M,+)$ and a cyclic group of order $p$. Moreover, the centre $\operatorname{Z}(B,+)$ is an ideal of~$B$ as it is a characteristic subgroup of $(B,+)$. If $p$ divides $|\operatorname{Z}(B,+)|$, then $(B,+)$ is abelian and, by Lemma~\ref{nilptype}, we arrive to a contradiction. Therefore, $\operatorname{Z}(B,+) \leq M$. Thus, either $\operatorname{Z}(B,+) = M$ or $\operatorname{Z}(B,+) = 0$. The former case cannot occur, since it would imply that $(B,+)$ is abelian. Hence, $\operatorname{Z}(B,+) = 0$.

On the other hand, it also holds that $(B,\cdot)$ is a semidirect product of $(M,\cdot)$ and a cyclic group of order~$p$. We write $(B,\cdot) = M\langle y\rangle_{\bigcdot}$ for some $y\in B$ with multiplicative order $o_{\bigcdot}(y) = p$. We claim that the order of $y$ in the additive group $(B,+)$ is also~$p$. Indeed, the quotient $B/M$ is trivial of order $p$, and therefore, $p = o_{\bigcdot}(yM) = o_{+}(y+M)$, which implies that $p$ divides $o_{+}(y)$. If a prime $q$ divides $o_{+}(y)$, then $0\neq ry \in M$ for positive integer $r$, and therefore, $\operatorname{Z}(B,+)\neq 0$, which is not possible. Hence, $o_+(y) = p$. Consequently, $B = M + \langle y \rangle_{\bigcdot} = M + \langle y \rangle_+$.

Our next step is to prove that $\operatorname{Z}(B,\cdot)$ is an ideal of~$B$. Thus, we need to verify that $\operatorname{Z}(B,\cdot)$ is $\lambda$-invariant and $\operatorname{Z}(B,\cdot)\trianglelefteq (B,+)$. If $p$ divides $|\operatorname{Z}(B,\cdot)|$, then $(B,\cdot)$ is abelian, and we arrive to a contradiction by Lemma~\ref{SolvSB}. Therefore, $\operatorname{Z}(B,\cdot) \le M$.

Let $b\in B$ and $a\in \operatorname{Z}(B,\cdot)$. Since $\operatorname{Z}(B,\cdot) \le M$, $\lambda_b(a) \in M$, which is abelian, and therefore, it suffices to show that $\lambda_b(a)$ commutes with~$y$. We write $b = m+ ky$ for some $m\in M$ and $k\in \Z$. Thus, we see that
\[
\lambda_b(a)=-b+ba= -(m+ky)+(m+ky)a
= -ky -m+ma-a+ (ky)a.
\]
Since $a,m\in M$, we have $ma=m+a$, and therefore, 
\[\lambda_b(a)= -ky +(ky)a \in \operatorname{C}_{(B,\bigcdot)}(y),\]
which is a subbrace by Lemma~\ref{lema:centralsubbrace}.

We now prove that $\operatorname{Z}(B,\cdot)\trianglelefteq (B,+)$. Let $a\in \operatorname{Z}(B,\cdot)$ and $b\in B$, and write $b = m + ky$ for some $k\in \mathbb{Z}$. Since $M$ is an abelian ideal and applying Lemma~\ref{lema:centralsubbrace}, 
\[ -b+a+b = -ky+a+ky\in M \cap \operatorname{C}_{(B,\bigcdot)}(y).\]
Thus, $-b+a+b\in \operatorname{Z}(B,\cdot)$. Therefore, $\operatorname{Z}(B,\cdot)$ is an ideal of~$B$ contained in~$M$. Thus, either $\operatorname{Z}(B,\cdot)=0$ or $\operatorname{Z}(B,\cdot)=M$. The case $\operatorname{Z}(B,\cdot)=M$ would imply that $(B,\cdot)$ is abelian and we arrive to a contradiction applying Lemma~\ref{SolvSB}. Hence, $\operatorname{Z}(B,\cdot)=0$.

Finally, observe that $p$ divides the order of the subbrace $\operatorname{C}_{(B,\bigcdot)}(y)$. By the minimality of $B$, $\operatorname{C}_{(B,\bigcdot)}(y) = B$ which contradicts the fact that $\operatorname{Z}(B,\cdot) = 0$. Hence, this final contradiction proves the theorem.
\end{proof}

\section{Proof of Theorem~\ref{teo:bi-skew}}
We start by recalling some well-known facts on bi-skew braces (see~\cite{biSkewbrace}, for example). If $(B,+,\cdot)$ is a skew brace $(B,+,\cdot)$, then $(B,\cdot,+)$ is again a skew brace. If we denote $\gamma$ the lambda-map of $(B,\cdot,+)$, then for every $a,b\in B$, we have
\[
\gamma_{a^{-1}}(b)=a(a^{-1}+b)=-a+ab=\lambda_a(b).
\]
Thus, it also holds $\lambda_a\in\mathrm{Aut}(B,\cdot)$. Since $a+b = \lambda_{a^{-1}}(b)$ for every $a,b\in B$, it also holds that for every $b,c\in B$,
\[
\lambda_a(b)\lambda_a(c)
=\lambda_a(bc)
=\lambda_a\bigl(b+\lambda_b(c)\bigr)
=\lambda_a(b)+\lambda_a\lambda_b(c)
=\lambda_a(b)
\lambda_{\lambda_a(b)^{-1}}\,\lambda_a\lambda_b(c),
\]
which yields $\lambda_a(c) = \lambda_{\lambda_a(b)^{-1}ab}(c)$. Therefore, $\lambda_{\lambda_a(b)a}=\lambda_{ab}$ for every $a,b\in B$. Replacing $b$ by $\lambda_a^{-1}(b)$ gives
\[
\lambda_{ba}
=\lambda_{\lambda_a(\lambda_a^{-1}(b))a}
=\lambda_{a\lambda_a^{-1}(b)}
=\lambda_{a+b}.
\]
Hence, the $\lambda$-map in $(B,+,\cdot)$ can be seen as a skew brace homomorphism (it~conserves sums and products) between $(B,+,\cdot)$ and the almost trivial skew brace $(\mathrm{Aut}(B,+),\circ^{\mathrm{opp}},\circ)$, where $\circ$ denotes composition in $\Aut(B,+)$. As a consequence, $\ker \lambda$ is an ideal of $B$ with quotient an almost trivial skew brace. Hence, $(B^{\opp})^{2} \le \ker \lambda$.


\begin{Lemma} \label{simpleBiSkewBrace}
 Let $B$ be a simple bi-skew brace. Then $B$ is either trivial or almost trivial.   
\end{Lemma}

\begin{proof}
   Suppose that $B$ is not almost trivial. Take $a \in (B^{\opp})^{2}$ and $b \in B$. It follows that $a*b=-a+ab-b=\lambda_{a}(b)-b=0$, as $(B^{\opp})^{2} \le \ker \lambda$. Thus $(B^{\opp})^{2}*B=0$. Since $B$ is simple and $B$ is not almost trivial we get $B*B=0$, i.e $B$ is a trivial skew brace.
\end{proof}

\begin{proof}[Proof of Theorem \ref{teo:bi-skew}]
Let $B$ be a minimal counterexample and take $p$ a prime dividing $|B|$ such that $B$ has no subbrace of order~$p$. Observe that $B$ is neither simple nor trivial nor almost trivial, by Lemma~\ref{simpleBiSkewBrace} and Remark~\ref{Note:trivialSemitrivial}. 

Let $M$ be a minimal ideal of $B$. Arguing as in Lemma~\ref{SolvSB}, we conclude that $|B| = p|M|$ and $p$ does not divide~$|M|$. In particular, we have that the quotient $B/M$ is an abelian skew brace isomorphic to a cyclic group of order~$p$. It follows that $(B^{\mathrm{opp}})^2 \leq M$, and by the minimality of~$M$, we obtain $(B^{\mathrm{opp}})^2 = M$ as $B$ is not almost trivial. Moreover, it also holds that $M = (B^{\mathrm{opp}})^2 \leq \ker(\lambda)$. 

Observe that $\ker \lambda$ is a subbrace of~$B$. Then, either $\ker \lambda = B$ or $\ker \lambda = M$. The former case is not possible, as it would imply that $B$ is trivial. Thus, we must have $M = \ker(\lambda)$.

Take $P \in \mathrm{Syl}_p(B,+)$. For every $m\in M$ and every $b\in B$, we have that $\lambda_{bm}(P) = \lambda_b(P) \in \mathrm{Syl}_p(B,+)$. Therefore, the action
\[
\Phi \colon B/M \longrightarrow \mathrm{Sym}\bigl(\mathrm{Syl}_p(B,+)\bigr), 
\qquad
bM \longmapsto \bigl(P \mapsto \lambda_b(P)\bigr),
\]
is well defined. Since $\bigl|\mathrm{Syl}_p(B,+)\bigr|$ divides $|M|$ and $p$ does not divide $|M|$ there exists $P \in \mathrm{Syl}_p(B,+)$ such that $\lambda_b(P) = P$, for all $b \in B$. In particular, $P$ is a subbrace of~$B$ of order~$p$, and we arrive to a final contradiction.
\end{proof}

\section*{Acknowledgements}
The authors would like to thank Andrea Loi for several helpful comments and suggestions which improved the presentation of this work. The first author is a member of the INdAM–GNSAGA research group. The second author is supported by the Spanish government: Ministerio de Ciencia, Innovación y Universidades, Agencia Estatal, FEDER (grant PID2024-159495NB-I00), and the Conselleria d'Educació, Universitats i Ocupació, Generalitat Valenciana (grant: \mbox{CIAICO/2023/007}).

\bibliographystyle{plain}
\bibliography{refrenze}

@article{Acri_2020,
   title={Skew braces of size {$pq$}},
   volume={48},
   ISSN={1872--1881},
   url={http://dx.doi.org/10.1080/00927872.2019.1709480},
   DOI={10.1080/00927872.2019.1709480},
   number={5},
   journal={Comm. Algebra},
   publisher={Informa UK Limited},
   author={E. Acri and M. Bonatto},
   year={2020}, 
   pages={1872–1881} }

@misc{dededind-skew-brace25,
      title={On {D}edekind {S}kew {B}races}, 
      author={A. Caranti and I. Del Corso and M. Di Matteo and M. Ferrara and M. Trombetti},
      year={2025},
      note = {Preprint: arXiv 2507.23550},
      eprint={2507.23550},
      archivePrefix={arXiv}, 
}

@Article{AcriBonatto20-algcoll-ab,
  author = 	 {E. Acri and M. Bonatto},
  title = 	 {Skew braces of order {$p^2q$} {I}: Abelian type},
  journal = 	 {Alg. Colloq.},
  year = 	 2020,
  volume = 	 29,
  number = 	 2,
  pages = 	 {297--320},
  doi =          {10.1142/S1005386722000244}
}

@Article{AcriBonatto22-jaa-nonab,
  author = 	 {E. Acri and M. Bonatto},
  title = 	 {Skew braces of size {$p^2q$} {II}: {N}on-abelian type},
  journal = 	 {J. Algebra Appl.},
  volume = 21,
  number = 3,
  year = 	 2022,
  pages = {2250062},
  doi =          {10.1142/S0219498822500621}}

@article{ballesterbolinches2024solubleskewleftbraces,
      title={Soluble skew left braces and soluble solutions of the {Y}ang-{B}axter equation}, 
      author={A. Ballester-Bolinches and R. Esteban-Romero and P. Jiménez-Seral and V. Pérez-Calabuig},
      year={2024},
      journal={Adv. Math.},
      volume={455},
      pages = {109880}, 
}

@article{BallesterEstebanKurdachenkoPerezC-25-actions,
      title={From actions from an abelian group on itself to left braces},
      author={A. Ballester-Bolinches and R. Esteban-Romero and L.A. Kurdachenko and V. Pérez-Calabuig},
      year=2025,
      journal = {Math.\ Proc.\ Camb.\  Philos.},
      pages = {65--79},
      volume = 178,
      number = 1,
}

@article{byott2015,
      title={Solubility Criteria for {H}opf-{G}alois Structures}, 
      author={N. P. Byott},
      year={2015},
      journal={New York J. Math.},
      volume={21},
      pages={883-903},
}

@article{Byott24-insoluble,
    author = {N. P. Byott},
    title = {On insoluble transitive subgroups in the holomorph of a finite soluble group},
    journal = {J. Algebra},
    year = 2024,
    volume = 638,
    pages = {1--31},
}

@article{DameleLoi25,
    author = {M. Damele and A. Loi},
    title = {Structural and rigidity properties of {L}ie skew braces}, 
    journal = {J. Algebra},
    note = {In press.},
    year = {2025}
}

@article{childs19,
      title={Bi-skew braces and {H}opf {G}alois structures}, 
      author={L. N. Childs},
      year={2019},
      journal={New York J. Math.},
      volume={25},
      pages={574--588},
}

@book{Kourovka26,
  address =	 {Novosibirsk, Russia},
  edition =	 21,
  editor =	 {V. D. Mazurov and E. I. Khukhro},
  publisher =	 {Russian Academy of Sciences, Siberian Branch,
                  Institute of Mathematics},
  title =	 {Unsolved problems in Group Theory: The {K}ourovka
                  Notebook},
  year =	 2026
}

@article{nasybullov2019,
      title={Connections between properties of the additive and the multiplicative groups of a two-sided skew brace}, 
      author={T. Nasybullov},
      year={2019},
      journal={J. Algebra},
      volume = 540,
      pages = {156--167},
}

@article{Tsang_2019,
   title={On the solvability of regular subgroups in the holomorph of a finite solvable group},
   volume={30},
   ISSN={1793-6500},
   url={http://dx.doi.org/10.1142/S0218196719500735},
   DOI={10.1142/s0218196719500735},
   number={2},
   journal={Int. J. Algebra Comput.},
   publisher={World Scientific Pub Co Pte Lt},
   author={C. Tsang and C. Qin},
   year={2019}, 
   pages={253–265} }

@article{Trappeniers_2023,
   title={On two-sided skew braces},
   volume={631},
   ISSN={0021-8693},
   url={http://dx.doi.org/10.1016/j.jalgebra.2023.05.003},
   DOI={10.1016/j.jalgebra.2023.05.003},
   journal={J. Algebra},
   publisher={Elsevier BV},
   author={Trappeniers, S.},
   year={2023},
   pages={267–286} }

@article{Rump20,
  author = {W. Rump},
  title = {One-generator braces and indecomposable set-theoretic solutions to the {Y}ang-{B}axter equation},
  journal = {Proc. Edinburgh Math. Soc.},
  volume = 63,
  year = 2020,
  pages = {676--696},
}

@article{RUMP2007153,
title = {Braces, radical rings, and the quantum Yang–Baxter equation},
journal = {J. Algebra},
volume = {307},
number = {1},
pages = {153-170},
year = {2007},
issn = {0021-8693},
doi = {https://doi.org/10.1016/j.jalgebra.2006.03.040},
url = {https://www.sciencedirect.com/science/article/pii/S0021869306002626},
author = {W. Rump}
}

@article{koch2020oppositeskewleftbraces,
      title={Opposite skew left braces and applications}, 
      author={Alan Koch and Paul J. Truman},
      year={2020},
      journal={J. Algebra},
      volume = 546,
      pages = {218--235},
}

@article{biSkewbrace,
   title={On bi-skew braces and brace blocks},
   volume={227},
   ISSN={0022-4049},
   url={http://dx.doi.org/10.1016/j.jpaa.2022.107295},
   DOI={10.1016/j.jpaa.2022.107295},
   number={5},
   journal={J.\ Pure\ Appl.\ Algebra},
   publisher={Elsevier BV},
   author={L. Stefanello and S. Trappeniers},
   year={2023}, 
   pages={107295} }

\end{document}